\documentclass[11pt]{amsart}

\usepackage{a4wide, amsmath, amsfonts, amssymb, mathrsfs, amsthm}
\usepackage[hyperfootnotes=false]{hyperref}
\usepackage[dvipsnames]{xcolor}

\allowdisplaybreaks[2]

\numberwithin{equation}{section}

\newtheorem{theorem}{Theorem}[section]
\newtheorem{lemma}[theorem]{Lemma}
\newtheorem{corollary}[theorem]{Corollary}
\newtheorem{remark}[theorem]{Remark}
\newtheorem{proposition}[theorem]{Proposition}
\newtheorem{assumption}[theorem]{Assumption}
\newtheorem{definition}[theorem]{Definition}

\newcommand{\1}{\mathbf{1}}
\newcommand{\dd}{\,\mathrm{d}}

\newcommand{\cA}{\mathcal{A}}
\newcommand{\E}{\mathbb{E}}
\newcommand{\R}{\mathbb{R}}
\newcommand{\N}{\mathbb{N}}
\renewcommand{\P}{\mathbb{P}}
\newcommand{\cB}{\mathcal{B}}
\newcommand{\cF}{\mathcal{F}}
\newcommand{\cN}{\mathcal{N\mkern-2mu N}}
\renewcommand{\epsilon}{\varepsilon}
\renewcommand{\rho}{\varrho}

\DeclareMathOperator{\spn}{span}

\title[Universal approximation property of neural SDEs]{Universal approximation property of neural stochastic differential equations}

\author[Kwossek]{Anna P. Kwossek}
\address{Anna P. Kwossek, University of Mannheim, Germany}
\email{anna.kwossek@uni-mannheim.de}

\author[Pr{\"o}mel]{David J.~Pr{\"o}mel}
\address{David J. Pr{\"o}mel, University of Mannheim, Germany}
\email{proemel@uni-mannheim.de}

\author[Teichmann]{Josef Teichmann}
\address{Josef Teichmann, ETH Zurich, Switzerland}
\email{jteichma@math.ethz.ch}

\date{\today}

\begin{document}
	
\begin{abstract}
  We identify various classes of neural networks that are able to approximate continuous functions locally uniformly subject to fixed global linear growth constraints. For such neural networks the associated neural stochastic differential equations can approximate general stochastic differential equations, both of It{\^o} diffusion type, arbitrarily well. Moreover, quantitative error estimates are derived for stochastic differential equations with sufficiently regular coefficients.
\end{abstract}
	
\maketitle
	
	
	
	

\section{Introduction}

Modeling approaches that hybridize the notion of differential equations with neural networks have recently become of interest, see~\cite{E2017,Chen2018}. In particular, neural stochastic differential equations (neural SDEs) have emerged as a powerful mathematical tool for capturing complex dynamical systems that exhibit randomness, see~\cite{Tzen2019,Jia2019,Kidger2021}. Specifically, these are stochastic differential equations in which neural networks are used to parametrize the drift and diffusion coefficient, thus extending the notion of neural ordinary differential equations. Neural SDEs have been successfully applied to develop data-driven methods for modeling, learning, and generating random dynamics due to powerful training technologies. For instance, they serve as continuous-time generative models for irregular time series, see~\cite{Liu2019,Li2020,Kidger2021,Issa2024}, and, notably, as very tractable and universal models for financial markets, thus being of particular interest for financial engineering, see~\cite{Cuchiero2020,Gierjatowicz2022,Cohen2022,Cohen2023,Choudhary2023,Fan2024}. In other words: neural stochastic differential equations constitute a continuous time counterpart of recurrent neural networks.

What motivates many of these applications is the key insight that neural stochastic differential equations are, at least, expected to approximate general SDEs arbitrarily well, thus providing fairly general and flexible models for stochastic processes and time series, such as recurrent neural networks approximate generic discrete dynamics. In fact, classical universal approximation theorems for neural networks, as proven, e.g., in~\cite{Cybenko1989,Hornik1991}, state that neural networks approximate any continuous function arbitrarily well uniformly on compact subsets of $\R^d$ or in an $L^p$-sense globally on $\R^d$. Hence, it seems intuitively reasonable that neural SDEs would inherit the universality of neural networks, allowing them to approximate generic SDEs (under mild regularity conditions). However, classical universal approximation theorems do not guarantee any uniform control of the global growth of the involved neural networks and, therefore, do not rigorously imply a universal approximation property for the associated neural SDEs. 

In this paper, we provide a theoretical justification for the universality of neural SDEs: in Section~\ref{sec: UAT with growth constraint} we identify various classes of neural networks that have the so-called `universal approximation property under a linear growth constraint', that is, are able to approximate continuous functions locally uniformly subject to a given global linear growth constraint. Exemplary classes of neural networks with this universal approximation property include single hidden layer feed-forward neural networks with linearly activating activation functions, such as logistic sigmoid and hyperbolic tangent, and deep feed-forward neural networks combining rather general activation functions with rectified linear unit (ReLU) activation functions. For the proof of these universal approximation theorems with global constraints, we rely on universal approximation theorems on weight spaces, as proven in~\cite{Cuchiero2024}, as well as on $L^p$-spaces, as proven in \cite{Kidger20}, and extend some of the methods of both works.

In Section~\ref{sec: UAP of neural SDEs} we demonstrate that the `universal approximation property under a linear growth constraint' of neural networks guarantees the universality of the associated neural SDEs. Indeed, assuming that an SDE possesses a unique solution, this solution can be approximated arbitrarily well by solutions of neural SDEs in a standard $L^2$-norm for stochastic processes if the involved neural networks do satisfy the `universal approximation property under a linear growth constraint'. Moreover, we derive quantitative error estimates for the approximation of stochastic differential equations with coefficients that fulfill standard conditions such as Lipschitz and H{\"o}lder continuity.

\medskip
\noindent\textbf{Organization of the paper:} In Section~\ref{sec: UAT with growth constraint}, we derive the `universal approximation property under a linear growth constraint' for various classes of neural networks. For these, we prove in Section~\ref{sec: UAP of neural SDEs} that the associated neural SDEs can approximate general SDEs.

\medskip
\noindent\textbf{Acknowledgment:} D.~J.~Pr{\"o}mel and A.~P.~Kwossek gratefully acknowledge financial support through the ``Eliteprogramm f{\"u}r Postdocs'' funded by the Baden-W{\"u}rttemberg Stiftung.

\section{Universal approximation property under a linear growth constraint}\label{sec: UAT with growth constraint}
	
In this section, we identify various classes of neural networks allowing for the approximation of continuous functions locally uniformly subject to a given linear growth constraint.
	
\medskip
	
We start by precisely formulating the aforementioned approximation property. The spaces $\R^k$ and $\R^{n_1 \times n_2}$ are equipped with the Euclidean norm $|\cdot|$. Let $C([0,T] \times \R^k; \R^{n_1 \times n_2})$ be the set of continuous functions $f \colon [0,T] \times \R^k \to \R^{n_1 \times n_2}$. Given a set $K \subset [0,T] \times \R^k$ and $f \in C([0,T] \times \R^k; \R^{n_1 \times n_2})$, we define
\begin{equation*}
  \|f\|_{\infty, K} := \sup_{x\in K} |f(x)|.
\end{equation*}
Moreover, we write $C(\R^d; \R^e)$ for the space of continuous maps $f \colon \R^d \to \R^e$,  $C^0_b(\R^d; \R^e)$ for the space of bounded and continuous functions $f \colon \R^d \to \R^e$, and $C^\infty(\R^d; \R^e)$ for the space of smooth functions $f \colon \R^d \to \R^e$, i.e.~functions with all its derivatives up to arbitrary order being continuous.
	
\begin{definition}\label{def: approximation property}
  A set $\cN \subset C([0,T] \times \R^k; \R^{n_1 \times n_2})$ is said to have the \textup{universal approximation property under a linear growth constraint} if the following property holds:
		
  For every function $f\in C([0,T] \times \R^k; \R^{n_1 \times n_2})$ with at most linear growth, i.e., there exists a constant $C_f > 0$ such that
  \begin{equation*}
     |f(t,x)| \leq C_f (1+|x|), \qquad t\in [0,T], \; x \in \R^k,
  \end{equation*}
  for every $\epsilon \in (0,1)$ and every compact set $K\subset \R^k$, there exists a function $\varphi \in \cN$ such that
  \begin{equation*}
    \|\varphi - f\|_{\infty, [0,T] \times K} \leq \epsilon,
  \end{equation*}
  and there exists a constant $\widetilde{C}_f>0$, not depending on $\epsilon$ and $K$, such that
  \begin{equation*}
	|\varphi(t,x)| \leq \widetilde{C}_f (1+|x|), \qquad t\in [0,T], \; x \in \R^k.
  \end{equation*}
\end{definition}
	
In the following four subsections we provide various classes of neural networks satisfying the universal approximation property under a linear growth constraint.
	
\subsection{Linearly activating activation functions}
	
To introduce the first class of neural networks that have the universal approximation property under a linear growth constraint we rely on the notion of weighted spaces as introduced in \cite{Cuchiero2024} in the context of neural networks. To that end, we fix the weight function
\begin{equation*}
  \psi \colon \R^{k+1} \to (0, \infty), \qquad \psi(x) := 1 + |x|, \qquad x \in \R^{k+1}.
\end{equation*}
The pre-image $\psi^{-1}((0,r])$ is compact in $\R^{k+1}$, for any $r > 0$, and hence, $\psi$ is an admissible weight function and $(\R^{k+1},\psi)$ is a weighted space in the sense of \cite[Section~2.1]{Cuchiero2024}. We further introduce the weighted norm $\|\cdot\|_{\cB_\psi(\R^{k+1};\R^{n_1 \times n_2})}$ as
\begin{equation*}
  \|f\|_{\cB_\psi(\R^{k+1};\R^{n_1 \times n_2})} := \sup_{x \in \R^{k+1}} \frac{|f(x)|}{\psi(x)},
\end{equation*}
for $f \colon \R^{k+1} \to \R^{n_1 \times n_2}$ such that $\sup_{x \in \R^{k+1}} \frac{|f(x)|}{\psi(x)} < \infty$. The space $\cB_\psi(\R^{k+1};\R^{n_1 \times n_2})$ is the weighted function space defined as the $\|\cdot\|_{\cB_\psi(\R^{k+1};\R^{n_1 \times n_2})}$-closure of $C^0_b(\R^{k+1}; \R^{n_1 \times n_2})$. Note that $\cB_\psi(\R^{k+1};\R^{n_1 \times n_2})$ is a separable Banach space when equipped with the norm $\|\cdot\|_{\cB_\psi(\R^{k+1};\R^{n_1 \times n_2})}$, which contains $C^0_b(\R^{k+1}; \R^{n_1 \times n_2})$, whereas $C^0_b(\R^{k+1}; \R^{n_1 \times n_2})$ is of course not separable with respect to the uniform norm.
	
Given an activation function $\rho \in C(\R;\R)$, a \emph{single hidden layer (feed-forward) neural network} $\varphi \colon \R^{n_0} \to \R^{n_1 \times n_2}$ is defined by
\begin{equation}\label{eq: neural network}
  \varphi(x) = \sum_{n=1}^N w_n \rho (a_n^\top x + b_n),
\end{equation}
for $x \in \R^{n_0}$, where $N \in \N$ denotes the number of neurons, where $w_1, \ldots, w_N \in \R^{n_1 \times n_2}$, $a_1, \ldots, a_N \in \R^{n_0}$ and $b_1, \ldots, b_N \in \R$ denote the linear readouts, weight vectors and biases, respectively. For $\rho \in C(\R;\R)$, we denote by $\cN^\rho_{n_0; n_1 \times n_2}$ the set of neural networks of the form~\eqref{eq: neural network} with activation function $\rho$.

Following \cite[Definition~4.3]{Cuchiero2024}, an activation function $\rho \in C(\R;\R)$ is called \emph{linearly activating} if $\cN^\rho_{1;1 \times 1} \subseteq \cB_\psi(\R;\R)$ and $\cN^\rho_{1;1 \times 1}$ is dense in $\cB_\psi(\R;\R)$.
	
\begin{remark}
  An activation function $\rho \in C(\R;\R)$ is linearly activating if it holds that $\lim_{x \to \pm \infty} \frac{|\rho(ax+b)|}{\psi(x)} = 0$ for any $a \in \N_0$, $b \in \R$, and $\rho$ is sigmoidal, i.e., $\lim_{x \to -\infty} \rho(x)= 0$ and $\lim_{x \to \infty} \rho(x) = 1$, see~\cite[Proposition~4.4]{Cuchiero2024}. Examples include the logistic sigmoid $\rho(x) = \frac{1}{1+\exp(-x)}$ and $\rho(x) = \tanh(x)$. Other conditions for   activation functions to be linearly activating are the discriminatory property or conditions on its Fourier transform, which can be found in~\cite[Proposition~4.4]{Cuchiero2024}.
\end{remark}
	
For the single hidden layer neural networks $\cN^\rho_{k+1; n_1 \times n_2}$ with linearly activating activation function~$\rho$, we obtain the following universal approximation theorem allowing for given a linear growth constraint.
	
\begin{theorem}
  If the activation function $\rho \in C(\R;\R)$ is linearly activating, then $\cN^\rho_{k+1; n_1 \times n_2}$ has the universal approximation property under a linear growth constraint in the sense of Definition~\ref{def: approximation property}. Moreover, the constant $\widetilde{C}_f$ in Definition~\ref{def: approximation property} can be chosen to be $\widetilde{C}_f = (1 + T)(1 + C_f)$.
\end{theorem}
	
\begin{proof}
  Let $f\in C([0,T] \times \R^k; \R^{n_1 \times n_2})$ be such that there exists a constant $C_f > 0$ satisfying
  \begin{equation}\label{eq: linear growth}
	|f(t,x)| \leq C_f (1+|x|), \qquad t \in [0,T], \; x \in \R^k.
  \end{equation}
  \emph{Step 1.} We extend $f$ to $\R^{k+1}$ by setting $f(t,x) := f(0,x)$, $t \leq 0$, and $f(t,x) := f(T,x)$, $t \geq T$. Given some function $g \in C^\infty(\R;\R)$ with compact support, $g \colon \R \to [0,1]$, $g(t) = 1$ for $t \in [0,T]$, we now consider $\tilde{f}(t,x) := f(t,x) g(t)$. Note that \eqref{eq: linear growth} holds for $\tilde{f}$, which implies that $\|\tilde{f}\|_{\cB_\psi(\R^{k+1};\R^{n_1 \times n_2})} \leq C_f$.
		
  \emph{Step 2.}
  Suppose that $\epsilon \in (0,1)$ and $K \subset \R^k$ is a compact set. Now there exists $\tilde{f}_{\epsilon,K} \in \cB_\psi(\R^{k+1}; \R^{n_1 \times n_2})$ satisfying~\eqref{eq: linear growth},
  \begin{equation*}
    f(t,x) = \tilde{f}_{\epsilon,K}(t,x), \qquad t \in [0,T], \; x \in K,
  \end{equation*}
  and
  \begin{equation*}
     \|\tilde{f}_{\epsilon,K}\|_{\cB_\psi(\R^{k+1}; \R^{n_1 \times n_2})} \leq \|\tilde{f}\|_{\cB_\psi(\R^{k+1}; \R^{n_1 \times n_2})} \leq C_f.
  \end{equation*}
  More precisely, take $\tilde{g} \in C^\infty(\R^{k+1};\R)$ with compact support, $\tilde{g} \colon \R^{k+1} \to [0,1]$, $\tilde{g}(t,x) = 1$ for $t \in [0,T]$, $x \in K$, and set $\tilde{f}_{\epsilon,K} := \tilde{f} \tilde{g}$. Then $\tilde{f}_{\epsilon,K} \in \cB_\psi(\R^{k+1};\R^{n_1 \times n_2})$.
		
  \emph{Step 3.} Let $\mathcal{H} \subseteq \cB_\psi(\R^{k+1};\R)$ be the additive family given by
  \begin{equation*}
    \mathcal{H} = \{x \mapsto a^\top x + b \, : \, a \in \R^{k+1}, b \in \R\}
  \end{equation*}
  see~\cite[Definition~4.1, Example~4.2]{Cuchiero2024}. We note that any $\varphi \in \cN^\rho_{k+1; n_1 \times n_2}$ is of the form
  \begin{equation*}
    \varphi(x) = \sum_{n=1}^N w_n \rho(h_n(x)),
  \end{equation*}
  where $h_1, \ldots, h_N \in \mathcal{H}$, and $\sup_{x \in \R^{k+1}} \frac{\psi(h(x))}{\psi(x)} < \infty$, for all $h \in \mathcal{H}$. Then, \cite[Theorem~4.13]{Cuchiero2024} gives that $\cN^\rho_{k+1; n_1 \times n_2}$ is dense in $\cB_\psi(\R^{k+1};\R^{n_1 \times n_2})$, i.e., there exists $\varphi \in \cN^\rho_{k+1; n_1 \times n_2}$ with
  \begin{equation*}
    \|\varphi - \tilde{f}_{\epsilon,K}\|_{\cB_\psi(\R^{k+1};\R^{n_1 \times n_2})} \leq \epsilon \bigg(\sup_{(t,x) \in [0,T] \times K} \psi((t,x))\bigg)^{-1}.
  \end{equation*}
  This implies that
  \begin{equation*}
    \|\varphi - f\|_{\infty, [0,T] \times K} = \|\varphi - \tilde{f}_{\epsilon,K}\|_{\infty, [0,T] \times K} \leq \frac{\epsilon}{2},
  \end{equation*}
  and
  \begin{align*}
    |\varphi(t,x)|
	& \leq (\|\varphi - \tilde{f}_{\epsilon,K}\|_{\cB_\psi(\R^{k+1}; \R^{n_1 \times n_2})} + \|\tilde{f}_{\epsilon,K}\|_{\cB_\psi(\R^{k+1}; \R^{n_1 \times n_2})}) \psi((t,x)) \\
	&\leq (1 + C_f)(1 + T)(1 + |x|),
  \end{align*}
  for $t \in [0,T]$, $x \in \R^k$. 
  
  Therefore the universal approximation result on $\cB_\psi(\R^{k+1};\R^{n_1 \times n_2})$ implies the universal approximation property under a linear growth constraint in the sense of Definition~\ref{def: approximation property}.
\end{proof}
	
\subsection{Combining the ReLU activation function and a general activation function}

To allow for an activation function that is not linearly activating, such as the widely used rectified linear unit (ReLU) activation function $\rho(x) := \max(x,0)$, we consider a different neural network architecture.
	
Let $L, N_0, \ldots, N_L \in \N$, and for any $l \in \{1, \ldots, L\}$, let $w_l \colon \R^{N_{l-1}} \to \R^{N_l}$, $x \mapsto A_l x + b_l$, be an affine function with $A_l \in \R^{N_l \times N_{l-1}}$ and $b_l \in \R^{N_l}$. Given an activation function $\rho \in C(\R;\R)$, a \emph{deep (feed-forward) neural network} $\varphi \colon \R^{N_0} \to \R^{N_L}$ is defined by
\begin{equation*}
  \varphi = w_L \circ \rho \circ w_{L-1} \circ \ldots \circ \rho \circ w_1,
\end{equation*}
where $\circ$ denotes the usual composition of functions. Here, $\rho$ is applied componentwise, $L-1$ denotes the number of hidden layers ($L$ is the depth of $\varphi$), and $N_1, \ldots, N_{L-1}$ denote the dimensions (widths) of the hidden layers and $N_0$ and $N_L$ the dimension of the input and the output layer, respectively. 
	
We write $\cN^\rho_{N_0;N_L}$ for the set of deep feed-forward neural networks $\varphi \colon \R \to \R$ with activation function $\rho$, input dimension $N_0$ and output dimension $N_L$ and an arbitrary number of hidden layers $L$, see e.g.~\cite[Appendix~B.1]{Cuchiero2020}. We then write $\cN^\rho_{n_0;n_1,n_2}$ for the set of functions $\varphi \colon \R^{n_0} \to \R^{n_1 \times n_2}$ of the form $\varphi = (\varphi^{ij})_{i = 1, \ldots, n_1, \, j = 1, \ldots, n_2}$, where $\varphi^{ij} \in \cN^{\rho}_{n_0;1}$.
	
When allowing for two activation functions $\rho_1, \rho_2 \in C(\R;\R)$, we write $\cN^{\rho_1,\rho_2}_{N_0;N_L}$ and $\cN^{\rho_1,\rho_2}_{n_0;n_1,n_2}$, respectively.
	
\begin{proposition}\label{prop: UAP with ReLU in weighted space}
  If $\rho_1 \colon \R \to \R$ is non-affine continuous and continuously differentiable at at least one point, with non-zero derivative at that point, and $\rho_2$ is the ReLU activation function, then $\cN^{\rho_1,\rho_2}_{k+1;n_1,n_2}$ has the universal approximation property under a linear growth constraint. Moreover, the constant $\widetilde{C}_f$ in Definition~\ref{def: approximation property} can be chosen to be $\widetilde{C}_f = \sqrt{n_1 n_2} (1+T) (1 + C_f)$.
\end{proposition}
	
\begin{remark}
  The condition on $\rho_1$ in Proposition~\ref{prop: UAP with ReLU in weighted space} is rather mild. For instance, it is satisfied by the frequently used activation functions, and it even includes polynomials. Furthermore, one may also consider both $\rho_1$ and $\rho_2$ to be the ReLU activation function.
\end{remark}
	
\begin{proof}[Proof of Proposition~\ref{prop: UAP with ReLU in weighted space}] 
  We first shall prove that $\cN^{\rho_1,\rho_2}_{n_0;1}$ is dense in $\cB_\psi(\R^{n_0};\R)$, $n_0 \in \N$, i.e., for every $f \in \cB_\psi(\R^{n_0};\R)$ and $\epsilon > 0$ there exists some $\varphi \in \cN^{\rho_1,\rho_2}_{n_0;1}$ such that
  \begin{equation}\label{eq: deep networks are dense in weighted space}
	\|\varphi - f\|_{\cB_\psi(\R^{n_0};\R)} = \sup_{x \in \R^{n_0}} \frac{|f(x) - \varphi(x)|}{\psi(x)} < \epsilon.
  \end{equation}
  In this proof, we adapt the methods of~\cite{Cuchiero2024}.
        
  \emph{Step 1.} A vector space $\cA$ of maps $a \colon \R^{n_0} \to \R$ is called a \emph{subalgebra} if $\cA$ is closed under multiplication, i.e., for every $a_1, a_2 \in \cA$ it holds that $a_1 \cdot a_2 \in \cA$. Moreover, $\cA$ is called \emph{point separating} if for every distinct $x_1, x_2 \in \R^{n_0}$, there exists some $a \in \cA$ with $a(x_1) \neq a(x_2)$. $\cA$ \emph{vanishes nowhere} if for every $x \in \R^{n_0}$, there exists some $a \in \cA$ with $a(x) \neq 0$. 
		
  For a given subalgebra $\cA \subseteq C(\R^{n_0};\R)$, a vector subspace $\mathcal{W} \subseteq C(\R^{n_0};\R)$ is called an \emph{$\cA$-submodule} if $a \cdot w \in \mathcal{W}$, for all $a \in \cA$ and $w \in \mathcal{W}$, where $x \mapsto (a \cdot w)(x) := a(x) w(x)$. 
		
  We consider the additive family $\mathcal{H} \subseteq \cB_\psi(\R^{n_0};\R)$ given by
  \begin{equation*}
	\mathcal{H} = \{x \mapsto a^\top x + b \, : \, a \in \R^{n_0}, b \in \R\},
  \end{equation*}
  see~\cite[Definition~4.1, Example~4.2]{Cuchiero2024}, and define $\cA := \spn(\{\cos \circ \, h : h \in \mathcal{H}\} \cup \{\sin \circ \, h : h \in \mathcal{H}\})$. It follows from~\cite[part~(ii) of Lemma~2.7]{Cuchiero2024} that $\cA \subseteq \cB_\psi(\R^{n_0};\R)$ since $(\cos \circ \, h) \vert_{K}$, \sloppy $(\sin \circ \, h) \vert_{K} \in C(K;\R)$, for all $h \in \mathcal{H}$ and compact subsets $K \subset \R$, and $\cos \circ \, h, \sin \circ \, h \in C^0_b(\R^{n_0};\R)$. Moreover, we note that $\cA$ is a subalgebra of $\cB_\psi(\R^{n_0};\R)$. Further, we define the subset $\mathcal{W} := \{\R^{n_0} \ni x \mapsto a(x) y \in \R : a \in \cA, y \in \R\} \subseteq \cB_\psi(\R^{n_0};\R)$, which is a vector subspace (as $\cA \subseteq \cB_\psi(\R^{n_0};\R)$ and $\R$ are both vector (sub)spaces) and an $\cA$-submodule by definition. 
		
  \emph{Step 2.} 
  We observe that $\cA \subseteq \cB_\psi(\R^{n_0};\R)$ vanishes nowhere as $(x \mapsto a(x) := \cos(0) = 1) \in \cA$. Moreover, $\cA \subseteq \cB_\psi(\R^{n_0};\R)$ is point separating and consists only of bounded maps.
		
  Hence, $\mathcal{W}$ is dense in $\cB_\psi(\R^{n_0};\R)$ by the weighted vector valued Stone--Weierstrass theorem~\cite[Theorem~3.8]{Cuchiero2024}.
		
  \emph{Step 3.} In this step we show that for every $f \in C^0_b(\R;\R)$ and $\epsilon > 0$, there exists $\varphi \in \cN^{\rho_1,\rho_2}_{1;1}$ such that
  \begin{equation*}
	\sup_{z \in \R} \frac{|\varphi(z) - f(z)|}{\psi(z)} < \epsilon.
  \end{equation*}
  We use this result in Step~4 to show that $\mathcal{W}$ is contained in the $\|\cdot\|_{\cB_\psi(\R^{n_0};\R)}$-closure of $\cN^{\rho_1,\rho_2}_{n_0;1}$, which then gives~\eqref{eq: deep networks are dense in weighted space}.
		
  Suppose that $f \in C^0_b(\R;\R)$ and $\epsilon > 0$, and define the constant $C := \frac{\epsilon}{3} + \sup_{z \in \R} |f(z)|$. Choose $r > 0$ large enough such that $r \geq 3 C \epsilon^{-1}$, and set $K_r := \psi^{-1}((0,r])$, which is a compact subset in $\R$. Since $\cN^{\rho_1}_{1;1}$ is dense in $C(\R;\R)$ with respect to the locally uniform norm, see~\cite[Proposition~3.12]{Kidger20}, there exists some $\varphi \in \cN^{\rho_1}_{1;1}$ such that
  \begin{equation*}
    \sup_{z \in K_r} |\varphi_1(z) - f(z)| < \frac{\epsilon}{3},
  \end{equation*}
  which implies that $|\varphi_1(z)| \leq C$ for all $z \in K_r$.
		
  Let $g \in C^0_b(\R;\R)$ be the function defined by $g(z) = \min(\max(z,-C),C)$, for $z \in \R$. Thus $g(\varphi_1(z)) = \varphi_1(z)$, for all $z \in K_r$. Then we get that
  \begin{equation*}
	\sup_{z \in \R} \frac{|g(\varphi_1(z)) - f(z)|}{\psi(z)} \leq \sup_{z \in K_r} |\varphi_1(z) - f(z)| + \sup_{z \in \R \setminus K_r} \frac{|g(\varphi_1(z))|}{\psi(z)} + \sup_{z \in \R \setminus K_r} \frac{|f(z)|}{\psi(z)} < \frac{\epsilon}{3} + \frac{2C}{R} \leq \epsilon. 
  \end{equation*}
  We now note that $\R^2 \ni (x,y) \mapsto \max(x,y) = \rho_2(x - y) + y$ and $\R^2 \ni (x,y) \mapsto \min(x,y) = x - \rho_2(x-y)$. This gives that there exists $\varphi \in \cN^{\rho_1,\rho_2}_{1;1}$, by adding two more hidden layers, calculating
  \begin{equation*}
	\varphi(z) := -\rho_2(-\rho_2(\varphi_1(z) + C) + 2 C) + C = \min(\max(\varphi_1(z),-C),C) = g(\varphi_1(z)).
  \end{equation*}
  \emph{Step 4.} In this step we verify that $\mathcal{W}$ is contained in the $\|\cdot\|_{\cB_\psi(\R^{n_0};\R)}$-closure of $\cN^{\rho_1,\rho_2}_{n_0;1}$.
		
  Suppose that $\epsilon > 0$, $h \in \mathcal{H}$, and $y \in \R$. We can assume without loss of generality that $y \neq 0$. Moreover, we consider the finite constant $C_h := \sup_{x \in \R^{n_0}} \frac{\psi(h(x))}{\psi(x)} + 1 > 0$. By Step~3, there exists some $\varphi \in \cN^{\rho_1,\rho_2}_{n_0;1}$ such that
  \begin{equation*}
	\sup_{z \in \R} \frac{|\varphi(z) - \cos(z)|}{\psi(z)} < \frac{\epsilon}{C_h |y|}.
  \end{equation*}
		
  Now, for the function $(x \mapsto w(x) := \cos(h(x))y) \in \mathcal{W}$, we define $(x \mapsto \phi(x):= y \varphi(h(x)))$, which is an element of $\cN^{\rho_1,\rho_2}_{n_0;1}$. Then we have that
  \begin{align*}
	&\|\phi - w\|_{\cB_\psi(\R^{n_0};\R)} = \sup_{x \in \R^{N_0}} \frac{|y \varphi(h(x)) - y \cos(h(x))|}{\psi(x)} \\
	&\quad \leq |y| \sup_{x \in \R^{n_0}} \frac{|\varphi(h(x)) - \cos(h(x))|}{\psi(x)} \\
	&\quad \leq |y| \sup_{x \in \R^{n_0}} \frac{\psi(h(x))}{\psi(x)} \sup_{x \in \R^{n_0}} \frac{|\varphi(h(x)) - \cos(h(x))|}{\psi(h(x))} \\
	&\quad \leq C_h |y| \sup_{z \in \R} \frac{|\varphi(z) - \cos(z)|}{\psi(z)} \\
    &\quad < \epsilon.
  \end{align*}
  Since $\epsilon$ was chosen arbitrarily, the map $(x \mapsto w(x) = \cos(h(x)) y) \in \mathcal{W}$ belongs to the $\|\cdot\|_{\cB_\psi(\R^{n_0};\R)}$-closure of $\cN^{\rho_1,\rho_2}_{n_0;1}$, which holds analogously true for $(x \mapsto \sin(h(x))y) \in \mathcal{W}$. Hence, due to the trigonometric identities for the product of cosine and sine, the entire $\cA$-submodule $\mathcal{W}$ is contained in the $\|\cdot\|_{\cB_\psi(\R^{n_0};\R)}$-closure of $\cN^{\rho_1,\rho_2}_{n_0;1}$.
	
		Since $\mathcal{W}$ is dense in $\cB_\psi(\R^{n_0};\R)$ by Step~2, we obtain that $\cN^{\rho_1,\rho_2}_{n_0;1}$ is dense in $\cB_\psi(\R^{n_0};\R)$, that is,~\eqref{eq: deep networks are dense in weighted space} does hold.
		
		\emph{Step 5.} It remains to show that~\eqref{eq: deep networks are dense in weighted space} implies the universal approximation property  under a linear growth constraint in the sense of Definition~\ref{def: approximation property}. 
		
		Let $f\in C([0,T] \times \R^k;\R^{n_1 \times n_2})$ be such that there exists a constant $C_f>0$ satisfying
		\begin{equation}\label{eq: linear growth in weighted space}
			|f(t,x)| \leq C_f (1+|x|), \qquad t \in [0,T], \; x \in \R^k.
		\end{equation}
		We extend $f$ to $\R^{k+1}$ by setting $f(t,x) := f(0,x)$, $t \leq 0$, and $f(t,x) := f(T,x)$, $t \geq T$. Given some function $g \in C^\infty(\R;\R)$ with compact support, $g \colon \R \to [0,1]$, $g(t) = 1$ for $t \in [0,T]$, we now consider $\tilde{f}(t,x) := f(t,x) g(t)$. Note that \eqref{eq: linear growth in weighted space} holds for $\tilde{f}$, which implies that $\|\tilde{f}\|_{\cB_\psi(\R^{k+1};\R^{n_1 \times n_2})} \leq C_f$.
		
		\emph{Step 6.}
		Suppose that $\epsilon \in (0,1)$ and $K\subset \R^k$ is a compact set.
		Now there exists $\tilde{f}_{\epsilon,K}\in \cB_\psi(\R^{k+1};\R^{n_1 \times n_2})$ satisfying~\eqref{eq: linear growth in weighted space},
		\begin{equation*}
			f(t,x) = \tilde{f}_{\epsilon,K}(t,x), \qquad t \in [0,T], \; x \in K,
		\end{equation*}
		and
		\begin{equation*}
			\|\tilde{f}_{\epsilon,K}\|_{\cB_\psi(\R^{k+1}; \R^{n_1 \times n_2})} \leq \|\tilde{f}\|_{\cB_\psi(\R^{k+1}; \R^{n_1 \times n_2})} \leq C_f.
		\end{equation*}
		More precisely, take $\tilde{g} \in C^\infty(\R^{k+1};\R)$ with compact support, $\tilde{g} \colon \R^{k+1} \to [0,1]$, $\tilde{g}(t,x) = 1$ for $t \in [0,T]$, $x \in K$, and set $\tilde{f}_{\epsilon,K} := \tilde{f} \tilde{g}$. Then $\tilde{f}_{\epsilon,K} \in \cB_\psi(\R^{k+1};\R^{n_1 \times n_2})$.
		
		\emph{Step 7.} We write $f = (f^{ij})_{i = 1, \ldots, n_1, \, j = 1, \ldots, n_2}$, similarly for $\tilde{f}_{\epsilon,K}$, and let $\delta = \frac{\epsilon}{\sqrt{n_1n_2}}$. Then we infer from~\eqref{eq: deep networks are dense in weighted space} that there exist $\varphi^{ij} \in \cN^{\rho_1,\rho_2}_{k+1;1}$, $i = 1, \ldots, n_1$, $j = 1, \ldots, n_2$, such that
		\begin{equation*}
			\|\varphi^{ij} - \tilde{f}_{\epsilon,K}^{ij}\|_{\cB_\psi(\R^{k+1};\R)} \leq \delta \bigg(\sup_{(t,x) \in [0,T] \times K} \psi((t,x))\bigg)^{-1}.
		\end{equation*}
		This implies that
		\begin{equation*}
			\|\varphi^{ij} - f^{ij}\|_{\infty, [0,T] \times K} = \|\varphi^{ij} - \tilde{f}_{\epsilon,K}^{ij}\|_{\infty, [0,T] \times K} \leq \delta,
		\end{equation*}
		and
		\begin{align*}
			|\varphi^{ij}(t,x)|
			& \leq (\|\varphi^{ij} - \tilde{f}_{\epsilon,K}^{ij}\|_{\cB_\psi(\R^{k+1};\R)} + \|\tilde{f}_{\epsilon,K}^{ij}\|_{\cB_\psi(\R^{k+1};\R)}) \psi((t,x)) \\
			&\leq (1 + C_f)(1 + T)(1 + |x|),
		\end{align*}
        for $t \in [0,T]$, $x \in \R^k$. Therefore there exists $\varphi = (\varphi^{ij})_{i = 1, \ldots, n_1, \, j = 1, \ldots, n_2} \in \cN^{\rho_1,\rho_2}_{k+1; n_1,n_2}$ satisfying
		\begin{equation*}
			\|\varphi - f\|_{\infty,[0,T] \times K} \leq \epsilon, \qquad |\varphi(t,x)| \leq \sqrt{n_1 n_2} (1+C_f) (1+T) (1+|x|), \qquad t \in [0,T], \; x \in \R^k,
		\end{equation*}
		which concludes the proof.
	\end{proof}
	
In the course of the proof of Proposition~\ref{prop: UAP with ReLU in weighted space}, we have shown a universal approximation property on the weighted space $\cB_\psi(\R^{n_0};\R)$.
	
\begin{corollary}\label{cor: UAT with ReLU in weighted space}
  If $\rho_1 \colon \R \to \R$ be non-affine continuous and continuously differentiable at at least one point, with non-zero derivative at that point, and $\rho_2$ be the ReLU activation function, then $\cN^{\rho_1,\rho_2}_{n_0;1}$ is dense in $\cB_\psi(\R^{n_0};\R)$, i.e., for every $f \in \cB_\psi(\R^{n_0};\R)$ and $\epsilon > 0$ there exists some $\varphi \in \cN^{\rho_1,\rho_2}_{n_0;1}$ such that
  \begin{equation*}
    \|f - \varphi\|_{\cB_\psi(\R^{n_0};\R)} = \sup_{x \in \R^{n_0}} \frac{|f(x) - \varphi(x)|}{\psi(x)} < \epsilon.
  \end{equation*}
\end{corollary}

\begin{remark}
  A universal approximation property on general weighted spaces has been proven in~\cite[Theorem~4.13]{Cuchiero2024}, by lifting a universal approximation property of one-dimensional neural networks to an infinite dimensional setting. In our setting, we notice that it suffices to have an approximation property on $C^0_b(\R;\R)$ with respect to the weighted norm, and it is a sufficient but not necessary condition that the one-dimensional neural networks be a subset of and dense in $\cB_\psi(\R;\R)$. This allows us to handle activation functions that are not linearly activating, but requires considering deep neural networks and the ReLU activation function instead of single hidden layer neural networks.
\end{remark}
	
\begin{remark}
  In Proposition~\ref{prop: UAP with ReLU in weighted space} and Corollary~\ref{cor: UAT with ReLU in weighted space}, we consider $\cN^{\rho_1,\rho_2}_{n_0;1}$ to be the generally defined class of neural networks. We note however that the functions that do appear here are more precisely linear combinations of neural networks of the form
  \begin{equation*}
	\R^{n_0} \ni x \mapsto -\rho_2(-\rho_2(\varphi_1 (h(x)) + C) + 2C) + C,
  \end{equation*}
  where $C > 0$, $h \in \mathcal{H} = \{\R^{n_0} \ni x \mapsto a^\top x + b \, : \, a \in \R^{n_0}, b \in \R\}$ and $\varphi_1 \in \cN^{\rho_1}_{1;1}$ is a deep feed-forward neural network with activation function $\rho_1$ and fixed width.
				
  The assumption on $\rho_1$ ensures that $\cN^{\rho_1}_{1;1}$ is dense in $C(\R;\R)$ with respect to the locally uniform norm. One may therefore relax this assumption and consider $\rho_1$ to be of the form $\rho(x) = \sin(x) + v(x) \exp(-x)$, for some $v \colon \R \to \R$ that is bounded, continuous and nowhere differentiable, so $\rho_1$ is also nowhere differentiable, see~\cite[Proposition~4.15]{Kidger20}.
		
  It is also possible to assume $\rho_1 \colon \R \to \R$ to be continuous and non-polynomial, and to consider $\varphi_1 \colon \R \to \R$ to be a deep neural network, where each hidden layer has two neurons with the identity activation function and one neuron with activation function $\rho_1$. These, again, are dense in $C(\R,\R)$ with respect to the locally uniform norm, see~\cite[Proposition~4.2]{Kidger20}.
\end{remark}
	
\subsection{The ReLU activation function}
	
We want to further examine the universal approximation property under a linear growth constraint for deep neural networks with the ReLU activation function. We present a constructive proof leading to a slightly stronger result compared to Corollary~\ref{cor: UAT plus bound with ReLU} in the sense that it shows that the constant $\widetilde{C}_f$ does not depend on $T$, and thus allows for approximation results uniformly in time.
	
\begin{proposition}\label{prop: UAP with ReLU}
  If $\rho$ be the ReLU activation function, then $\cN^\rho_{k+1; n_1, n_2}$ has the universal approximation property under a linear growth constraint. Moreover, the constant $\widetilde{C}_f$ in Definition~\ref{def: approximation property} can be chosen to be $\widetilde{C}_f = \sqrt{n_1 n_2} (1 + C_f)$.
\end{proposition}
	
	\begin{proof}
		We shall prove that for any $f \in C(\R^{n_0};\R)$, $n_0 \in \N$, for any $\delta \in (0,1)$ and $K \subset \R^{n_0}$ compact, there exists a neural network $\varphi \in \cN^\rho_{n_0; 1}$ such that
		\begin{equation}\label{eq: UAP with ReLU}
			\|\varphi - f\|_{\infty, K} \leq \delta \qquad \text{and} \qquad |\varphi(x)| \leq |f(x)| + \delta, \qquad x \in \R^{n_0}.
		\end{equation}
		Suppose $K \subset \R^{n_0}$ is a compact set and $\delta \in (0,1)$. Without loss of generality, we assume that $K = \prod_{i=1}^{n_0} [a_i,b_i]$, for some $a_i, b_i \in \R$, $i = 1, \ldots, n_0$. Set $c > 0$ and consider $J = \prod_{i=1}^{N_0} [a_i-c,b_i+c]$.
		
		The proof is similar in spirit to the proof of ~\cite[Theorem~4.16]{Kidger20}. Since $\cN^\rho_{n_0;1}$ is dense in $C(\R^{n_0};\R)$ with respect to the locally uniform norm, see \cite[Proposition~4.9]{Kidger20}, there exists $\varphi_1 \in \cN^\rho_{n_0;1}$ with fixed width $n_0+2$ such that
		\begin{equation}\label{eq: compact-dense}
			\|\varphi_1 - f\|_{\infty, J} \leq \delta.
		\end{equation}
		
		We begin by extending the definition of a neuron, for sake of notation: an enhanced neuron means the composition of an affine map with the activation function $\rho$ with another affine map, and we allow for affine combinations of enhanced neurons.
		In the proof of~\cite[Proposition~4.9]{Kidger20} and in the following, one may use that $x \mapsto \rho(x+N) - N$ equals the identity function for $N$ suitably large, that is, one enhanced neuron may exactly represent the identity function. This allows us, first, to record the inputs in every hidden layer (called in-register neurons) and, second, to preserve the values of the corresponding neurons in the preceding layer.
		
		In each layer of $\varphi_1$, the first $n_0$ neurons are the in-register neurons, then we have the neuron which bases its computations on the in-register neurons applying $\rho$, and finally, the out-register neuron, which we associate the output with.
		
		We now modify $\varphi_1$ and construct $\varphi \in \cN^\rho_{n_0;1}$, by removing the output layer and adding some more hidden layers ($3n_0+1$ to be precise) such that $\varphi$ equals $\varphi_1$ on $K$ and vanishes on $\R^{n_0} \setminus J$, thus \eqref{eq: UAP with ReLU} holds.
		
		To that end, we use that two layers of two enhanced neurons each may represent the continuous piecewise affine function $U_i \colon \R \to \R$, where $U_i(x) = 1$, $x \in [a_i,b_i]$, and $U_i(x) = 0$, $x \in (-\infty, a_i - c] \cup [b_i + c, \infty)$, $i = 1, \ldots, n_0$, see~\cite[Kidger~B.1]{Kidger20}. 
		
		Similarly, one layer of two enhanced neurons may represent $[0,\infty)^2 \ni (x,y) \mapsto \min(x,y)$, see~\cite[Lemma~B.2]{Kidger20}. 
		
		By adding $2n_0$ hidden layers, we are therefore able to store the values of $U_i(x_i)$, $i = 1, \ldots, n_0$, in the in-register neurons. By adding $n_0-1$ hidden layers, we are able to compute and store the value of $U$ in one of the in-register neurons, where
		\begin{equation*}
			U(x) := \min_{i = 1, \ldots, n_0} U_i(x_i),
		\end{equation*}
		which approximates the indicator function $\1_{K}$, mapping into $[0,1]$, with support in $J$, taking value $1$ on $K$, and value $0$ on $\R^{n_0} \setminus J$.
		
		It further holds that $\R^2 \ni (x,y) \mapsto \max(x,y) = \rho(x - y) + y$ and $\R^2 \ni (x,y) \mapsto \min(x,y) = x - \rho(x-y)$. Therefore there exists $\varphi \in \cN^\rho_{n_0;1}$, by adding two more hidden layers and the output layer, calculating
		\begin{equation*}
			\varphi := - \rho(-\rho(\varphi_1 + C U) + 2 C U) + C U = \min(\max(\varphi_1, -C U), C U),
		\end{equation*}
		for some suitable constant $C > 0$ depending only on $f$ and $J$ such that $|\varphi_1(x)| \leq C$ for any $x \in J$, see~\eqref{eq: compact-dense}. 
		
		By definition, it holds that $U(x) = 1$, $x \in K$, and $U(x) = 0$, $x \in \R^{n_0} \setminus J$, thus we deduce that
		\begin{equation}\label{eq: cut-off}
			\varphi(x) = \varphi_1(x), \qquad x \in K, \qquad \text{and} \qquad \varphi(x) = 0, \qquad x \in \R^{n_0} \setminus J.
		\end{equation}
		It then immediately follows from~\eqref{eq: compact-dense} that
		\begin{equation*}
			\|\varphi - f\|_{\infty, K} \leq \delta.
		\end{equation*}
		One can further verify that $|\varphi(x)| \leq |\varphi_1(x)|$, $x \in J$. Combining~\eqref{eq: compact-dense} and~\eqref{eq: cut-off}, we obtain that
		\begin{equation*}
			|\varphi(x)| \leq |f(x)| + \delta, \qquad x \in \R^{n_0}.
		\end{equation*}
		This proves~\eqref{eq: UAP with ReLU}. 
		
		We now show that this implies the universal approximation property with given linear  growth constraint.
		
		Let $f \in C([0,T] \times \R^k; \R^{n_1 \times n_2})$ be such that there exists a constant $C_f > 0$ satisfying 
		\begin{equation*}
			|f(t,x)| \leq C_f (1 + |x|), \qquad t \in [0,T], \; x \in \R^k.
		\end{equation*}
		We extend $f$ to $\R^{k+1}$ by setting $f(t,x) := f(0,x)$, $t \leq 0$, and $f(t,x) := f(T,x)$, $t \geq T$, and write $f = (f^{ij})_{i = 1, \ldots, n_1, \, j = 1, \ldots, n_2}$. Suppose that $K \subset \R^k$ is a compact set and $\epsilon \in (0,1)$, and let $\delta = \frac{\epsilon}{\sqrt{n_1 n_2}}$. Then we have shown that there exist $\varphi^{ij} \in \cN^\rho_{k+1;1}$, $i = 1, \ldots, n_1$, $j = 1, \ldots, n_2$, such that
		\begin{equation*}
			\|\varphi^{ij} - f^{ij}\|_{\infty,[0,T] \times K} \leq \delta \qquad \text{and} \qquad |\varphi^{ij}(t,x)| \leq (1 + C_f)(1 + |x|), \qquad t \in [0,T], \; x \in \R^k.
		\end{equation*}
		This implies that there exists $\varphi = (\varphi^{ij})_{i = 1, \ldots, n_1, \, j = 1, \ldots, n_2} \in \cN^\rho_{k+1; n_1,n_2}$ satisfying
		\begin{equation*}
			\|\varphi - f\|_{\infty,[0,T] \times K} \leq \epsilon, \qquad |\varphi(t,x)| \leq \sqrt{n_1 n_2} (1+C_f) (1+|x|), \qquad t \in [0,T], \; x \in \R^k,
		\end{equation*}
		which concludes the proof.
	\end{proof}
	
In the course of the proof, we have shown the following corollary, which implies the universal approximation property under a linear growth constraint.
	
\begin{corollary}\label{cor: UAT plus bound with ReLU}
  If $\rho$ be the ReLU activation function, then for any $f \in C(\R^{n_0};\R)$, for any $\epsilon \in (0,1)$ and $K \subset \R^{n_0}$ compact, there exists a neural network $\varphi \in \cN^\rho_{n_0;1}$ such that 
  \begin{equation*}
	\|\varphi - f \|_{\infty, K} \leq \epsilon \qquad \text{and} \qquad |\varphi(x)| \leq |f(x)| + \epsilon, \qquad x \in \R^{n_0}.
  \end{equation*}
\end{corollary}
	
\subsection{Two activation functions: the ReLU activation function and a squashing activation function}
    
When assuming two activation functions in the neural network architecture, a result analogous to Proposition~\ref{prop: UAP with ReLU} and Corollary~\ref{cor: UAT plus bound with ReLU} can be achieved. For this purpose, we introduce the notion of \emph{squashing} activation functions, i.e., monotone and sigmoidal functions, see~\cite{Hornik1991}. More precisely, $\rho \in C(\R;\R)$ is squashing, if $\rho$ is monotone, $\rho \colon \R \to [a,b]$, for some $a, b \in \R$, and $\lim_{x \to -\infty} \rho(x) = a$, $\lim_{x \to \infty} \rho(x) = b$. We assume without loss of generality that $a = 0$, $b = 1$.
	
\begin{proposition}\label{prop:UAT for ReLU and squashing}
  If $\rho_1 \in C(\R;\R)$ be squashing and continuous non-polynomial and continuously differentiable at at least one point, with non-zero derivative at that point, and $\rho_2$ be the ReLU activation function, then $\cN^{\rho_1, \rho_2}_{k+1; n_1, n_2}$ has the universal approximation property  under a linear growth constraint. Moreover, the constant $\widetilde{C}_f$ in Definition~\ref{def: approximation property} can be chosen to be $\widetilde{C}_f = \sqrt{n_1 n_2} (1+C_f)$.
\end{proposition}
	
\begin{remark}
  Examples for activation functions satisfying the assumptions of Proposition~\ref{prop:UAT for ReLU and squashing} are $\rho_1(x) = \frac{1}{1+\exp(-x)}$, $\rho_1(x) = \tanh(x)$ and $\rho_1(x) = \frac{x}{1+|x|}$.
\end{remark}
	
\begin{remark}
  One may relax the assumption that $\rho_1$ is squashing and assume that $\rho_1 \in C(\R;\R)$ be monotone and have one limit, either left or right. Then there exists $\tilde{\rho}_1 \in C(\R;\R)$ that is squashing, given as a composition of an affine map with $\rho_1$ with another affine map and $\rho_1$. This would allow to consider, e.g., $\rho_1(x) = \ln(1 + \exp(x))$.
\end{remark}
	
	\begin{proof}
		We shall prove that for any $f \in C(\R^{n_0};\R)$, $n_0 \in \N$, for any $\delta \in (0,1)$ and $K \subset \R^{n_0}$ compact, there exists a neural network $\varphi \in \cN^{\rho_1,\rho_2}_{n_0; 1}$ such that
        \begin{equation}\label{eq: UAP with ReLU and squashing}
			\|\varphi - f\|_{\infty, K} \leq \delta \qquad \text{and} \qquad |\varphi(x)| \leq |f(x)| + \delta, \qquad x \in \R^{n_0}.
		\end{equation}
		Suppose $K \subset \R^{n_0}$ is a compact set and $\delta \in (0,1)$. Without loss of generality, we assume that $K = \prod_{i=1}^{n_0} [a_i,b_i]$, for some $a_i, b_i \in \R$, $i = 1, \ldots, n_0$. Set $c > 0$ and consider $J = \prod_{i=1}^{N_0} [a_i-c,b_i+c]$.
		
		We follow the constructive proof of Proposition~\ref{prop: UAP with ReLU}. Since $\rho_1$ is assumed to be continuous non-polynomial and continuously differentiable at at least one point, with non-zero derivative at that point, $\cN^{\rho_1}_{n_0;1}$ is dense in $C(\R^{n_0};\R)$ with respect to the locally uniform norm, see \cite[Proposition~4.9]{Kidger20}. That is, there exists $\varphi_1 \in \cN^{\rho_1}_{n_0;1}$ (allowing the identity function in the output layer) with fixed width $n_0+2$ such that
		\begin{equation}\label{eq: compact-dense with squashing}
			\|\varphi_1 - f\|_{\infty,J} \leq \delta.
		\end{equation}
		(We note that $\rho_1$ may be replaced with $\rho_2$.) We begin by extending the definition of a neuron, for sake of notation: an enhanced neuron means the composition of an affine map with the activation function (here, $\rho_2$) with another affine map, and we allow for affine combinations of enhanced neurons. In the proof of~\cite[Proposition~4.9]{Kidger20} and in the following, one uses that $x \mapsto \rho_2(x+N) - N$ equals the identity function for $N$ suitably large, that is, one enhanced neuron may exactly represent the identity function. This allows us, first, to record the inputs in every hidden layer (called in-register neurons), and, second, to preserve the values of the corresponding neurons in the preceding layer.
		
		In each layer of $\phi$, the first $n_0$ neurons are the in-register neurons, then we have the neuron which bases its computations on the in-register neurons applying the activation function, and finally, we have the out-register neuron, which we associate the output with.
		
		We now modify $\varphi_1$ and construct $\varphi \in \cN^{\rho_1,\rho_2}_{n_0;1}$, by removing the output layer and adding some more hidden layers such that $\varphi$ equals $\varphi_1$ on $K$ and vanishes on $\R^{n_0} \setminus J$, thus \eqref{eq: UAP with ReLU and squashing} holds.
		
		We consider $\zeta > 0$ and set $\eta = \frac{1-\zeta}{2(n_0-1)+3}$. Then there exists some threshold $C_\eta > 0$ such that
		\begin{equation*}
			\rho_1(x) \in [0,\eta), \qquad x \leq - C_\eta, \qquad \text{and} \qquad \rho_1(x) \in (1-\eta,1], \qquad x \geq C_\eta.
		\end{equation*}
		We aim to find a neural representation of $\varphi_0 \colon \R \to [0,1]$ which takes values
		\begin{equation}\label{eq: squash}
			\varphi_0(x) \in (1-\eta,1], \qquad x \in K, \qquad \text{and} \qquad \varphi_0(x) \in [0,\eta), \qquad x \in \R^{n_0} \setminus J,
		\end{equation}
		using activation function $\rho_1$, and store the value of $\varphi_1(x)$ in one of the in-register neurons. Then we use that two layers of two enhanced neurons each, now using activation function ReLU, $\rho_2$, may represent the continuous piecewise affine function $U \colon \R \to \R$, where
		\begin{equation*}
			U(x) = 1, \qquad x \in [1-\eta,1], \qquad \text{and} \qquad U(x) = 0, \qquad x \in (-\infty, \eta] \cup [2(1-\eta),\infty),
		\end{equation*}
		see \cite[Lemma~B.1]{Kidger20}, noting that $\eta < 1-\eta < 1 < 2(1-\eta)$.
		
		We are therefore able to compute and store the value of $U_1(x)$ in one of the in-register neurons, where
		\begin{equation*}
			U_1(x) = 1, \qquad x \in K, \qquad \text{and} \qquad U_1(x) = 0, \qquad x \in \R^{n_0} \setminus J,
		\end{equation*}
		which approximates the indicator function $\1_K$.
		
		We proceed as in the proof of Proposition~\ref{prop: UAP with ReLU}: we add two more hidden layers and the output layer, with $\rho_2$, calculating
		\begin{equation*}
			\varphi(x) = \min(\max(\varphi_1,-CU_1),CU_1),
		\end{equation*}
		for some suitable constant $C > 0$ depending only on $f$ and $J$ such that $|\varphi_1(x)| \leq C$ for any $x \in J$, see~\eqref{eq: compact-dense with squashing}. It then follows that there exists $\varphi \in \cN^{\rho_1,\rho_2}_{n_0;1}$ which satisfies~\eqref{eq: UAP with ReLU and squashing} for $\frac{\delta}{2}$.
		
		The rest can be proven following the last paragraph in the proof of Proposition~\ref{prop: UAP with ReLU} verbatim.
		
		It remains to show~\eqref{eq: squash}. We make use of the squashing property of $\rho_1$ and get that one layer of two enhanced neurons may represent the function $h_i \colon \R \to [-1,1]$ that satisfies
		\begin{equation*}
			h_i(x) \in (1-2\eta,1], \quad x \in [a_i,b_i], \quad \text{and} \quad h(x) \in (-\eta,\eta), \quad x \in (-\infty,a_i-c] \cup [b_i+c,\infty),
		\end{equation*}
		namely
		\begin{equation*}
			h_i(x) = \rho_1(c_1(2x+c-2a_i)) - \rho_1(c_1(2x-c-2b_i)),
		\end{equation*}
		where $c_1 =  \frac{C_\eta}{c}$.
		
		We modify $h_i$ by $\tilde{h}_i \colon \R \to [2\eta - 2, 2\eta]$, $x \mapsto h_i(x) - (1 - 2 \eta)$, and it holds that
		\begin{equation*}
			\tilde{h}_i(x) \in (0,2\eta], \quad x \in [a_i,b_i], \quad \enspace \tilde{h}_i(x) \in (-(1-\eta),-(1-3\eta)), \quad x \in (-\infty, a_i - c] \cup [b_i + c, \infty).
		\end{equation*}
		This implies that
		\begin{equation*}
			\sum_{i=1}^{n_0} \tilde{h}_i(x_i) \in (0,\infty), \qquad x \in K, \qquad \text{and} \qquad \sum_{i=1}^{n_0} \tilde{h}_i(x_i) \in (-\infty, -\zeta), \qquad x \in \R^{n_0} \setminus J,
		\end{equation*}
		because if $x \in \R^{n_0} \setminus J$, there exists $i$ such that $x_i \in (-\infty, a_i - c) \cup (b_i + c, \infty)$, that is, $\sum_{i=1}^{n_0} \tilde{h}_i(x_i) \leq 2 \eta (n_0-1) - (1-3\eta) = - \zeta$.
        
		Lastly, since
		\begin{equation*}
			\rho_1(c_2(2x+\zeta)) \in [0,\eta), \qquad x \leq - \zeta, \qquad \text{and} \qquad \rho_1(c_2(2x+\zeta)) \in (1-\eta,1], \qquad x \geq 0,
		\end{equation*}
		for $c_2 = \frac{C_\eta}{\zeta}$, we consider
		\begin{equation*}
			\varphi_0(x) = \rho_1\Big(c_2\Big(2\sum_{i=1}^{n_0} \tilde{h}_i(x_i) + \zeta \Big) \Big),
		\end{equation*}
		which gives~\eqref{eq: squash}.
	\end{proof}

\section{Universal approximation property of neural SDEs}\label{sec: UAP of neural SDEs}
	
In this section, we derive a universal approximation property of neural stochastic differential equations (neural SDEs) assuming that the involved neural networks satisfy the universal approximation property under a linear growth constraint in the sense of Definition~\ref{def: approximation property}. We start by introducing the probabilistic framework.
	
\medskip
	
Let $T > 0$ be a fixed finite time horizon and let $W$ be a $d$-dimensional Brownian motion, defined on a probability space $(\Omega, \cF, \P)$ with a filtration $(\cF_t)_{t \in [0,T]}$ satisfying the usual conditions, i.e., completeness and right-continuity. Throughout this section, we consider the stochastic differential equation
\begin{equation}\label{eq: sde}
  X_t = x_0 + \int_0^t b(s,X_s) \dd s + \int_0^t \sigma(s,X_s) \dd W_s, \qquad t \in [0,T],
\end{equation}
where $x_0 \in \R^k$, $b \colon [0,T]\times \R^k \to \R^k $, $\sigma \colon [0,T]\times \R^k  \to \R^{k \times d}$ are continuous functions, and $\int_0^t \sigma(s,X_s) \dd W_s$ is defined as an It{\^o} integral. For a comprehensive introduction to stochastic It{\^o} integration and differential equations we refer, e.g., to the textbook \cite{Karatzas1991}. Moreover, we make the following assumption.

\begin{assumption}\label{ass:linear growth}
  Let $b \colon [0,T]\times \R^k \to \R^k $ and $\sigma \colon [0,T]\times \R^k  \to \R^{k \times d}$ be continuous functions such that
  \begin{equation*}
    |b(t,x)| + |\sigma(t,x)| \leq C_{b,\sigma} (1+ |x|), \qquad t\in [0,T], \; x \in \R^k,
  \end{equation*}
  for some constant $C_{b,\sigma}>0$.
\end{assumption}

In order to approximate the general SDE~\eqref{eq: sde}, we consider sets $\cN_1 \subset C([0,T]\times \R^k; \R^k)$ and $\cN_2 \subset C([0,T] \times \R^k; \R^{k \times d})$ having the universal approximation property  under a linear growth constraint. For $b_\epsilon \in \cN_1$ and $\sigma_{\epsilon} \in \cN_2$, the associated neural SDE is defined as
\begin{equation}\label{eq: neural sde}
  X^\epsilon_t = x_0 + \int_0^t b_\epsilon(s,X^\epsilon_s) \dd s + \int_0^t \sigma_\epsilon(s,X^\epsilon_s) \dd W_s, \qquad t \in [0,T].
\end{equation}
To ensure the existence of a unique solution~$X^\epsilon$ to the neural SDE~\eqref{eq: neural sde}, it is sufficient that $b_\epsilon$ and $\sigma_\epsilon$ are Lipschitz continuous with at most linear growth. Let $\textup{Lip}([0,T] \times \R^k; \R^{n_1 \times n_2})$ be the set of Lipschitz continuous functions $f\colon [0,T] \times \R^k \to \R^{n_1 \times n_2}$.
	
\begin{remark}
  The Lipschitz assumption on the neural networks is immediately satisfied if the underlying activation functions are Lipschitz continuous. Many frequently used activation functions are, indeed, Lipschitz continuous functions including ReLU, hyperbolic tangent, softsign, softplus and sigmoidal activation functions.
\end{remark}
	
Combining the universal approximation property  under a linear growth constraint and \cite[Theorem~A]{Kaneko1988}, we obtain the following universal approximation property of neural SDEs.
	
\begin{lemma}\label{lem:UAP of neural SDEs}
  Suppose Assumption~\ref{ass:linear growth} and that pathwise uniqueness holds for the SDE~\eqref{eq: sde}. Moreover, suppose that $\cN_1\subset\textup{Lip}([0,T] \times \R^k;\R^k)$ and $\cN_2\subset\textup{Lip}([0,T] \times \R^d;\R^{k \times d})$ have the universal approximation property  under a linear growth constraint in the sense of Definition~\ref{def: approximation property}. Let $K \subset \R^k$ be a compact set. Then for every $\epsilon >0$, there exist $b_{\epsilon}\in \cN_1$ and $\sigma_{\epsilon} \in \cN_2$ such that
  \begin{equation*}
	\sup_{ x_0 \in  K} \E \bigg[\sup_{t\in[0,T]} |X^{\epsilon,x_0}_t-X^{x_0}_t|^2\bigg] \leq \epsilon,
  \end{equation*}
  where $X^{x_0}$ and $X^{\epsilon,{x_0}}$ are the solutions to the SDE~\eqref{eq: sde} and the neural SDE~\eqref{eq: neural sde}, with initial value~$x_0$, respectively.
\end{lemma}

\begin{remark}
  The uniform linear growth condition, as required in the definition of the universal approximation property under a linear growth constraint (Definition~\ref{def: approximation property}), is a necessary condition for most approximation and stability results for stochastic differential equations, cf. \cite{Stroock2006,Friz2010,Mao2011}. For instance, assuming that the involved neural networks are real analytic, the flow of the associated neural stochastic differential equations is real analytic as well and can be used to approximate the flow of fairly general SDEs, see \cite{Ishiwata1999,Friz2010}. 
\end{remark}

Even though we do not apply the full strength of approximation in weighted spaces in this section, we still want to point out that, in contrast to Lemma \ref{lem:UAP of neural SDEs}, we can actually obtain a global approximation result of quantitative nature for the solutions of differential equations and their flows using weighted norms for the involved coefficients. We do only show it in the case of ordinary differential equations here and leave further investigations on weighted spaces and stochastic differential equations to future research.

\begin{lemma}
  Let $ V_i \colon \mathbb{R}^k \to \mathbb{R}^k$, $i=1,2$, be two $L$-Lipschitz continuous vector fields of at most linear growth, i.e., there exists a constant $L>0$ such that 
  \begin{equation*}
    | V_i (x)-V_i(y)| \leq L |x-y|
    \quad \text{and}\quad 
    | V_i (x)| \leq L (1+|x|),
  \end{equation*}
  for $x,y\in \R^k$. Let $\epsilon > 0$ and suppose that $ \| V_1 - V_2 \|_{\mathcal{B}_{\psi}(\R^k;\R^k)} \leq \epsilon$ with $\psi (x):=1+|x|$. Denote by $ X^i(x)$ the solution of
  \begin{equation*}
    X^i_t(x) = x + \int_0^t V_i(X_s^i(x)) \dd s, \qquad t\in [0,T],
  \end{equation*}
  with $X^i_0(x)=x$ for $i=1,2$ and $x \in \mathbb{R}^k$. Then, for every $T>0$ there is a constant $C > 0$ such that
  \begin{equation*}
    \sup_{t\in [0,T]} | X_t^1(x) -X_t^2(x) | \leq  2 \epsilon \max(1,LT) \exp(2LT) T \psi(x)
  \end{equation*}
  for all $x \in \mathbb{R}^k$.
\end{lemma}

\begin{proof}
  We can write
  \begin{align*}
    X_t^1(x) -X_t^2(x) &= \int_0^t V_1(X_s^1(x)) \dd s-\int_0^t V_2(X_s^2(x)) \dd s \\
    & = \int_0^t (V_1(X_s^1(x)) - V_1(X_s^2(x))) \dd s + \int_0^t (V_1(X_s^2(x)) - V_2(X_s^2(x))) \dd s \\
    & = \int_0^t \int_0^1 \nabla V_1(X_s^1(x) + \theta (X_s^1(x) - X_s^2(x))) \dd \theta \cdot (X_s^1(x) - X_s^2(x)) \dd s \\
    & +\int_0^t \frac{V_1(X_s^2(x)) - V_2(X_s^2(x))}{\psi(X_s^2(x))} \psi(X_s^2(x)) \dd s 
  \end{align*}
  for all $ x \in \mathbb{R}^k$ and $ t\in [0,T]$. Recall that, since $X^2$ is the solution of an ordinary differential equation with coefficient of at most linear growth, a straightforward application of Gronwall's inequality yields
  \begin{equation*}
    |X_t^2(x)| \leq \max(1,Lt) (1+|x|) \exp (Lt), \qquad t\in [0,T], \; x\in \R^k.
  \end{equation*}
  Hence, we obtain
    \begin{equation*}
    | X_t^1(x) -X_t^2(x) | \leq L \int_0^t | X_s^1(x) -X_s^2(x) | \dd s + \epsilon 2 \max(1,Lt) \exp (Lt) \int_0^t \psi(x) \dd s
  \end{equation*}
  for all $ x \in \mathbb{R}^k$ and $ t \in [0,T]$, which allows to conclude the claimed lemma by Gronwall's inequality.
\end{proof}
	
The universal approximation property provided in Lemma~\ref{lem:UAP of neural SDEs} ensures that general SDEs can be approximated arbitrary well by neural SDEs, assuming the corresponding neural networks satisfy the universal approximation property under a linear growth constraint. In the following subsection we deduce quantitative versions of these approximation results.
	
\subsection{Quantitative approximation results for SDEs with Lipschitz continuous coefficients}
	
For SDEs with Lipschitz continuous coefficients, we obtain the following quantitative approximation property of neural SDEs.
	
\begin{proposition}
  Let $p \geq 2$, suppose that Assumption~\ref{ass:linear growth} holds and that the coefficients $b$, $sigma$ of the SDE~\eqref{eq: sde} satisfy
  \begin{equation*}
    |b(t,x)-b(t,y)|+|\sigma (t,x)-\sigma (t,y)| \leq L_{b,\sigma}|x-y|, \qquad t\in [0,T], \; x,y \in \R^k, 
  \end{equation*}
  for some constant $L_{b,\sigma}>0$. Moreover, assume that $\cN_1\subset\textup{Lip}([0,T]\times\R^k;\R^k)$ and $\cN_2\subset\textup{Lip}([0,T]\times\R^k;\R^{k \times d})$ have the universal approximation property  under a linear growth constraint in the sense of Definition~\ref{def: approximation property}. Then for every $\epsilon > 0$, there exist $b_{\epsilon}\in \cN_1$ and $\sigma_{\epsilon} \in \cN_2$ satisfying
  \begin{equation*}
	\|b^{\epsilon} - b\|_{\infty, [0,T]\times K} +  \|\sigma^{\epsilon} - \sigma\|_{\infty, [0,T]\times K} \leq \delta,
  \end{equation*}
  where
  \begin{equation*}
	\delta^p := \frac{\epsilon}{2C} \exp(- C L_{b,\sigma}^2) \qquad \text{with} \qquad C = 2^{2(p-1)} T^{\frac{p}{2}} \Big(T^{\frac{p}{2}} + \Big(\frac{p^3}{2(p-1)}\Big)^{\frac{p}{2}} \Big),
  \end{equation*}
  and
  \begin{equation*}
	K := \{x \in \R^k \, : \, |x|^p \leq r\} \qquad \text{with}\qquad r:= \frac{2^{2p}}{\epsilon} (1+ 3^{2p-1}|x_0|^{2p})(\exp(\tilde{a}) + \exp(a)),
  \end{equation*}
  where
  \begin{equation*}
 	\tilde{a} := 6^{2p-1} \widetilde{C}_{b,\sigma}^{2p} T^p \Big(T^p + \frac{(2p)^{3p}}{2^p(2p-1)^p}\Big) \qquad \text{and} \qquad a := 6^{2p-1} C_{b,\sigma}^{2p} T^p \Big(T^p + \frac{(2p)^{3p}}{2^p(2p-1)^p}\Big),
  \end{equation*}
  where $\widetilde{C}_{b,\sigma} = \max(\widetilde{C}_b, \widetilde{C}_\sigma)$, and $\widetilde{C}_b$ and $\widetilde{C}_\sigma$ are given via Definition~\ref{def: approximation property}, such that
  \begin{equation*}
	\E \Big[ \sup_{t \in [0,T]} |X^\epsilon_t - X_t|^p \Big] \leq \epsilon,
  \end{equation*}
  where $X$ and $X^{\epsilon}$ are the solutions to the SDE~\eqref{eq: sde} and the neural SDE~\eqref{eq: neural sde}, respectively.
\end{proposition}
	
\begin{proof}
  First note that for any stochastic process $(Z_t)_{t \in [0,T]}$ and any stopping time $\tau \leq T$, we have that
  \begin{equation*}
	\E \Big[ \sup_{t \in [0,T]} |Z_t|^p \Big] \leq \bigg(\E \bigg[ \Big( \sup_{t \in [0,T]} |Z_t|^p \Big)^2 \bigg] \P(\tau < T) \bigg)^{\frac{1}{2}} + \E \Big[ \sup_{t \in [0,T]} |Z_{t \wedge \tau}|^p \Big].
  \end{equation*}
  Fixing $\epsilon > 0$ and setting $Z := X^\epsilon - X$, we aim to bound each of the summands on the right-hand side by $\frac{\epsilon}{2}$.
		
  By the estimate given in~\cite[Chapter~II, Theorem~4.4]{Mao2011}, we can bound
  \begin{equation*}
	\E \bigg[ \Big( \sup_{t \in [0,T]} |X^\epsilon_t - X_t|^p \Big)^2 \bigg] \leq 2^{2p-1} \Big( \E \Big[\sup_{t \in [0,T]} |X^\epsilon_t|^{2p} \Big] + \E \Big[ \sup_{t \in [0,T]} |X_t|^{2p} \Big] \Big) \leq \frac{\epsilon}{2} r.
  \end{equation*}
  Markov's inequality then implies that
  \begin{equation*}
	\P(\tau < T) \leq \frac{\epsilon}{2} r^{-1},
  \end{equation*}
  for the stopping time $\tau := \inf \{t \geq 0 : X^\epsilon_t \notin K\} \wedge T$. 
		
  Moreover, by Jensen's inequality and~\cite[Chapter~I, Theorem~7.2]{Mao2011}, we obtain that
  \begin{align*}
	&\E[\sup_{t \in [0,u]} |X^\epsilon_{t \wedge \tau} - X_{t \wedge \tau}|^p] \\
	&\quad \leq 2^{p-1} \E \Big[ \Big( \int_0^{u \wedge \tau} |b_\epsilon(s,X^\epsilon_s) - b(s,X_s)| \dd s \Big)^p \Big] \\
	&\qquad + 2^{p-1} \E \Big[ \Big( \sup_{t \in [0,u]} \Big | \int_0^{u \wedge \tau} (\sigma_\epsilon(s,X^\epsilon_s) - \sigma(s,X_s)) \dd W_s \Big| \Big)^p\Big] \\
	&\quad \leq (2T)^{p-1} \E \Big[ \int_0^{u \wedge \tau} |b_\epsilon(s,X^\epsilon_s) - b(s,X_s)|^p \dd s \Big] \\
	&\qquad + 2^{p-1} T^{\frac{p-2}{2}} \Big(\frac{p^3}{2(p-1)}\Big)^{\frac{p}{2}} \E \Big[ \int_0^{u \wedge \tau} |\sigma_\epsilon(s,X^\epsilon_s) - \sigma(s,X_s)|^p \dd s \Big] \\
	&\quad \leq 2^{2(p-1)} T^{p-1} \Bigg(\E \Big[ \int_0^{u \wedge \tau} |b_\epsilon(s,X^\epsilon_s) - b(s,X^\epsilon_s)|^p \dd s \Big] + \E \Big[ \int_0^{r \wedge \tau} |b(s,X^\epsilon_s) - b(s,X_s)|^p \dd s \Big] \Bigg) \\
   	&\qquad + 2^{2(p-1)} T^{\frac{p-2}{2}} \Big(\frac{p^3}{2(p-1)}\Big)^{\frac{p}{2}} \Big( \E \Big[ \int_0^{u \wedge \tau} |\sigma_\epsilon(s,X^\epsilon_s) - \sigma(s,X^\epsilon_s)|^p \dd s \Big] \\
	&\qquad \quad + \E \Big[ \int_0^{u \wedge \tau} |\sigma(s,X^\epsilon_s) - \sigma(s,X_s)|^p \dd s \Big]\Big) \\
	&\quad \leq \delta^p 2^{2(p-1)} T^{\frac{p}{2}} \Big( T^{\frac{p}{2}}  + \Big(\frac{p^3}{2(p-1)}\Big)^{\frac{p}{2}} \Big) \\
	&\qquad + 2^{2(p-1)} T^{\frac{p-2}{2}} \Big(T^{\frac{p}{2}} + \Big(\frac{p^3}{2(p-1)}\Big)^{\frac{p}{2}} \Big) L_{b,\sigma}^p  \int_0^u \E[\sup_{v \in [0,s]} |X^\epsilon_v - X_v|^p] \dd s,
	\end{align*}
	  for any $u \in [0,T]$. By Gr{\"o}nwall's inequality, it then holds that
	\begin{equation*}
      \E \Big[ \sup_{t \in [0,T]} |X^\epsilon_{t \wedge \tau} - X_{t \wedge \tau}|^p \Big] \leq \frac{\epsilon}{2}
	\end{equation*}
	since we have chosen $\delta$ accordingly. Combining the estimates thus concludes the proof.
\end{proof}
	
\subsection{Quantitative approximation results for SDEs with H{\"o}lder continuous diffusion coefficient}
	
In the one-dimensional case, the Lipschitz assumption on the diffusion coefficient~$\sigma$ in the SDE~\eqref{eq: sde} can be relaxed to H{\"o}lder continuity, leading to the following quantitative approximation property of neural SDEs.
	
\begin{proposition}
  Let $k = d = 1$, suppose that Assumption~\ref{ass:linear growth} holds and that the coefficients $b, \sigma$ of the SDE~\eqref{eq: sde} satisfy
  \begin{equation*}
	|b(t,x)-b(t,y)| \leq L_{b,\sigma}|x-y|
	\qquad \text{and}\qquad
	|\sigma (t,x)-\sigma (t,y)| \leq L_{b,\sigma}|x-y|^{\gamma},
  \end{equation*}
  for all $x,y\in \R^k$, $t\in [0,T]$ and for some constant $L_{b,\sigma}>0$ with $\gamma \in [\frac{1}{2},1]$. Moreover, assume that $\cN\subset\textup{Lip}([0,T]\times\R;\R)$ has the universal approximation property  under a linear growth constraint in the sense of Definition~\ref{def: approximation property}. Then for every $\epsilon > 0$, there exist $b_{\epsilon}, \sigma_\epsilon \in \cN$ satisfying
  \begin{equation*}
  	\|b^{\epsilon} - b\|_{\infty, [0,T]\times K} +  \|\sigma^{\epsilon} - \sigma\|_{\infty, [0,T]\times K} \leq \delta,
  \end{equation*}
  where $\alpha > 1$, $\beta > 0$ and $\delta \in (0,1)$ with  
  \begin{equation*}
	\bigg(\beta + \delta T + \frac{2 \alpha}{\beta \log(\alpha)} \delta^2 T + \frac{2 \alpha}{\log(\alpha)} \beta^{2\gamma - 1} L_{b,\sigma}^2 T\bigg)\exp( L_{b,\sigma} T)\leq \frac{\epsilon}{2},
  \end{equation*}
  and
  \begin{equation*}
    K := [-r,r] \qquad \text{with} \qquad r := \frac{4}{\epsilon}(1+ 3|x_0|^2)(\exp((24 T + 6 T^2) \widetilde{C}_{b,\sigma}^2) + \exp((24 T + 6 T^2) C_{b,\sigma}^2)),
  \end{equation*}
  where $\widetilde{C}_{b,\sigma}$ is given in Definition~\ref{def: approximation property}, such that
  \begin{equation*}
	\sup_{t \in [0,T]}  \E \big[ |X^\epsilon_t - X_t| \big] \leq \epsilon,
  \end{equation*}
  where $X$ and $X^{\epsilon}$ are the solutions to the SDE~\eqref{eq: sde} and the neural SDE~\eqref{eq: neural sde}, respectively.
\end{proposition}
	
\begin{proof}
  First note tat for any stochastic process $(Z_t)_{t \in [0,T]}$ and any stopping time $\tau \leq T$, we have that
  \begin{equation*}
	\E[|Z_t|]
	\leq (\E[|Z_t|^2] \P(\tau < T))^{\frac{1}{2}} + \E[|Z_{t \wedge \tau}|].
  \end{equation*}
  Fixing $\epsilon>0$, $t \in [0,T]$, and setting $Z = X^\epsilon - X$, we aim to bound each of the summands on the right-hand side by $\frac{\epsilon}{2}$.
		
  By the estimate given in~\cite[Chapter~II, Theorem~4.4]{Mao2011}, we can bound
  \begin{equation*}
    \E [|X^\epsilon_t - X_t|^2] \leq 2 \Big( \E \Big[\sup_{t \in [0,T]} |X^\epsilon_t|^2 \Big] + \E \Big[ \sup_{t \in [0,T]} |X_t|^2 \Big] \Big) \leq \frac{\epsilon}{2} r,
  \end{equation*}
  Markov's inequality then implies that
  \begin{equation*}
    \P(\tau < T) \leq \frac{\epsilon}{2} r^{-1},
  \end{equation*}
  for the stopping time $\tau := \inf \{t \geq 0 : X^\epsilon_t \notin K\} \wedge T$.

    We apply the idea of~\cite{Yamada71} to approximate $x \mapsto |x|$, see also~\cite{Gyongy11}. There exists $h \in C^2(\R)$ such that $|x| \leq \beta + h(x)$, $|h'(x)| \leq 1$, and $h''(x) \leq \frac{2}{|x| \log(\alpha)} \1_{[\frac{\beta}{\alpha},\beta]} (x)$. By It{\^o}'s formula, we then obtain that
    \begin{align*}
        |X^\epsilon_{t \wedge \tau} - X_{t \wedge \tau}|
        & \leq \beta + \int_0^{t \wedge \tau} h^\prime(X^\epsilon_s - X_s) (b^\epsilon(s,X^\epsilon_s) - b(s,X_s)) \dd s\\
        &\quad + \frac{1}{2} \int_0^{t \wedge \tau} h^{\prime\prime}(X^\epsilon_s - X_s) (\sigma^\epsilon(s,X^\epsilon_s) - \sigma_s(s,X_s))^2 \dd s \\
        &\quad + \int_0^{t \wedge \tau} h^\prime(X^\epsilon_s - X_s) (\sigma^\epsilon(s,X^\epsilon_s) - \sigma(s,X_s)) \dd W_s \\
        & \leq \beta + \delta T + L_{b,\sigma} \int_0^{t \wedge \tau} |X^\epsilon_s - X_s| \dd s \\
        &\quad + \frac{2 \alpha}{\beta \log(\alpha)} \delta^2 T + \frac{2 \alpha}{\beta \log(\alpha)} \beta^{2\gamma - 1} L_{b,\sigma}^2 T \\
        &\quad + \int_0^t h^\prime(X^\epsilon_s - X_s) (\sigma^\epsilon(s,X^\epsilon_s) - \sigma(s,X_s)) \dd W_s,
    \end{align*}
    for any $t \in [0,T]$. Let $M_t := \int_0^{t \wedge \tau} h^\prime(X^\epsilon_s - X_s) (\sigma^\epsilon(s,X^\epsilon_s) - \sigma(s,X_s)) \dd W_s$, $t \in [0,T]$. Since $\sigma$ and $\sigma_\epsilon$ are of linear growth and there exists some constant $C_0 > 0$ depending only on $C_{\sigma}$, $\widetilde{C}_\sigma$, $x_0$ and $T$ such that
    \begin{equation*}
        \E[|X_t|^2] + \E[|X^\epsilon_t|^2] \leq C_0^2, \qquad t \in [0,T],
    \end{equation*}
    see e.g.~\cite[Chapter~II, Corollary~4.6]{Mao2011}, it holds that $\E[[M]_t] < \infty$ for any $t \in [0,T]$, where $[M]$ denotes the quadratic variation of $M$. Hence, by \cite[Chapter~II.6, Corollary~3]{Protter05}, $M$ is a martingale. It follows that
    \begin{align*}
        &\E[|X^\epsilon_{t \wedge \tau} - X_{t \wedge \tau}|] \\
        &\quad \leq \beta + \delta T + \frac{2 \alpha}{\beta \log(\alpha)} \delta^2 T + \frac{2 \alpha}{\log(\alpha)} \beta^{2\gamma - 1} L_{b,\sigma}^2 T + L_{b,\sigma} \int_0^{t \wedge \tau} |X^\epsilon_s - X_s| \dd s
    \end{align*}
    and thus Gr{\"o}nwall's inequality yields that
    \begin{equation*}
        \E[|X^\epsilon_{t \wedge \tau} - X_{t \wedge \tau}|] 
        \leq \bigg(\beta + \delta T + \frac{2 \alpha}{\beta \log(\alpha)} \delta^2 T + \frac{2 \alpha}{\log(\alpha)} \beta^{2\gamma - 1} L_{b,\sigma}^2 T\bigg)\exp( L_{b,\sigma} T).
    \end{equation*}
    Choosing $\alpha$, $\beta$ and $\delta$ such that 
    \begin{equation*}
        \bigg(\beta + \delta T + \frac{2 \alpha}{\beta \log(\alpha)} \delta^2 T + \frac{2 \alpha}{\log(\alpha)} \beta^{2\gamma - 1} L_{b,\sigma}^2 T\bigg)\exp( L_{b,\sigma} T)\leq \frac{\epsilon}{2}
    \end{equation*}
    and combining the above estimates we conclude the proof.
\end{proof}

\bibliography{quellen}{}

\newcommand{\etalchar}[1]{$^{#1}$}
\providecommand{\bysame}{\leavevmode\hbox to3em{\hrulefill}\thinspace}
\providecommand{\MR}{\relax\ifhmode\unskip\space\fi MR }
\providecommand{\MRhref}[2]{%
  \href{http://www.ams.org/mathscinet-getitem?mr=#1}{#2}
}
\providecommand{\href}[2]{#2}
\begin{thebibliography}{GSVS{\etalchar{+}}22}

\bibitem[CJB23]{Choudhary2023}
Vedant Choudhary, Sebastian Jaimungal, and Maxime Bergeron, \emph{Fu{NV}ol: {A}
  {M}ulti-{A}sset {I}mplied {V}olatility {M}arket {S}imulator using
  {F}unctional {P}rincipal {C}omponents and {N}eural {SDE}s}, arXiv preprint
  arXiv:2303.00859 (2023).

\bibitem[CKT20]{Cuchiero2020}
Christa Cuchiero, Wahid Khosrawi, and Josef Teichmann, \emph{A generative
  adversarial network approach to calibration of local stochastic volatility
  models}, Risks \textbf{8} (2020), no.~4.

\bibitem[CRBD18]{Chen2018}
Ricky T.~Q. Chen, Yulia Rubanova, Jesse Bettencourt, and David Duvenaud,
  \emph{Neural ordinary differential equations}, Proceedings of the 32nd
  International Conference on Neural Information Processing Systems, 2018,
  pp.~6572--6583.

\bibitem[CRW22]{Cohen2022}
Samuel~N. Cohen, Christoph Reisinger, and Sheng Wang, \emph{Hedging option
  books using neural-{SDE} market models}, Appl. Math. Finance \textbf{29}
  (2022), no.~5, 366--401.

\bibitem[CRW23]{Cohen2023}
\bysame, \emph{Arbitrage-free neural-{SDE} market models}, Appl. Math. Finance
  \textbf{30} (2023), no.~1, 1--46.

\bibitem[CST24]{Cuchiero2024}
Christa Cuchiero, Philipp Schmocker, and Josef Teichmann, \emph{Global
  universal approximation of functional input maps on weighted spaces}, arXiv
  preprint arXiv:2306.03303 (2024).

\bibitem[Cyb89]{Cybenko1989}
G.~Cybenko, \emph{Approximation by superpositions of a sigmoidal function},
  Math. Control Signals Systems \textbf{2} (1989), no.~4, 303--314.

\bibitem[E17]{E2017}
Weinan E, \emph{A proposal on machine learning via dynamical systems}, Commun.
  Math. Stat. \textbf{5} (2017), no.~1, 1--11.

\bibitem[FS24]{Fan2024}
Lei Fan and Justin Sirignano, \emph{Machine learning methods for pricing
  financial derivatives}, arXiv preprint arXiv:2406.00459 (2024).

\bibitem[FV10]{Friz2010}
Peter~K. Friz and Nicolas~B. Victoir, \emph{Multidimensional stochastic
  processes as rough paths}, Cambridge Studies in Advanced Mathematics, vol.
  120, Cambridge University Press, Cambridge, 2010, Theory and applications.

\bibitem[GR11]{Gyongy11}
Istv\'an Gy\"ongy and Mikl\'os R\'asonyi, \emph{A note on {E}uler
  approximations for {SDE}s with {H}\"older continuous diffusion coefficients},
  Stochastic Process. Appl. \textbf{121} (2011), no.~10, 2189--2200.
  \MR{2822773}

\bibitem[GSVS{\etalchar{+}}22]{Gierjatowicz2022}
Patrick Gierjatowicz, Marc Sabate-Vidales, David Siska, Lukasz Szpruch, and Zan
  Zuric, \emph{Robust pricing and hedging via neural stochastic differential
  equations}, Journal of Computational Finance \textbf{26} (2022), no.~3,
  1--32.

\bibitem[Hor91]{Hornik1991}
Kurt Hornik, \emph{Approximation capabilities of multilayer feedforward
  networks}, Neural Networks \textbf{4} (1991), no.~2, 251--257.

\bibitem[IHLS24]{Issa2024}
Zacharia Issa, Blanka Horvath, Maud Lemercier, and Cristopher Salvi,
  \emph{Non-adversarial training of {N}eural {SDE}s with signature kernel
  scores}, Advances in Neural Information Processing Systems \textbf{36}
  (2024).

\bibitem[IT99]{Ishiwata1999}
Akira Ishiwata and Setsuo Taniguchi, \emph{On the analyticity of stochastic
  flows}, Osaka J. Math. \textbf{36} (1999), no.~1, 139--148.

\bibitem[JB19]{Jia2019}
Junteng Jia and Austin~R Benson, \emph{Neural jump stochastic differential
  equations}, Advances in Neural Information Processing Systems \textbf{32}
  (2019).

\bibitem[KFLL21]{Kidger2021}
Patrick Kidger, James Foster, Xuechen~(Chen) Li, and Terry Lyons,
  \emph{Efficient and accurate gradients for neural sdes}, Advances in Neural
  Information Processing Systems (M.~Ranzato, A.~Beygelzimer, Y.~Dauphin, P.S.
  Liang, and J.~Wortman Vaughan, eds.), vol.~34, Curran Associates, Inc., 2021,
  pp.~18747--18761.

\bibitem[KL20]{Kidger20}
Patrick Kidger and Terry Lyons, \emph{{Universal Approximation with Deep Narrow
  Networks}}, Proceedings of Thirty Third Conference on Learning Theory (Jacob
  Abernethy and Shivani Agarwal, eds.), Proceedings of Machine Learning
  Research, vol. 125, PMLR, 09--12 Jul 2020, pp.~2306--2327.

\bibitem[KN88]{Kaneko1988}
H.~Kaneko and S.~Nakao, \emph{A note on approximation for stochastic
  differential equations}, S\'{e}minaire de {P}robabilit\'{e}s, {XXII}, Lecture
  Notes in Math., vol. 1321, Springer, Berlin, 1988, pp.~155--162.

\bibitem[KS91]{Karatzas1991}
Ioannis Karatzas and Steven~E. Shreve, \emph{Brownian motion and stochastic
  calculus}, second ed., Graduate Texts in Mathematics, vol. 113,
  Springer-Verlag, New York, 1991.

\bibitem[LWCD20]{Li2020}
Xuechen Li, Ting-Kam~Leonard Wong, Ricky~TQ Chen, and David Duvenaud,
  \emph{Scalable gradients for stochastic differential equations},
  International Conference on Artificial Intelligence and Statistics, PMLR,
  2020, pp.~3870--3882.

\bibitem[LXS{\etalchar{+}}19]{Liu2019}
Xuanqing Liu, Tesi Xiao, Si~Si, Qin Cao, Sanjiv Kumar, and Cho-Jui Hsieh,
  \emph{Neural {SDE}: {S}tabilizing {N}eural {ODE} {N}etworks with {S}tochastic
  {N}oise}, arXiv preprint arXiv:1906.02355 (2019).

\bibitem[Mao08]{Mao2011}
Xuerong Mao, \emph{Stochastic differential equations and applications}, second
  ed., Horwood Publishing Limited, Chichester, 2008.

\bibitem[Pro05]{Protter05}
Philip~E. Protter, \emph{Stochastic integration and differential equations},
  Stochastic Modelling and Applied Probability, vol.~21, Springer-Verlag,
  Berlin, 2005, Second edition. Version 2.1, Corrected third printing.

\bibitem[SV06]{Stroock2006}
Daniel~W. Stroock and S.~R.~Srinivasa Varadhan, \emph{Multidimensional
  diffusion processes}, Classics in Mathematics, Springer-Verlag, Berlin, 2006,
  Reprint of the 1997 edition.

\bibitem[TR19]{Tzen2019}
Belinda Tzen and Maxim Raginsky, \emph{Neural {S}tochastic {D}ifferential
  {E}quations: {D}eep {L}atent {G}aussian {M}odels in the {D}iffusion {L}imit},
  arXiv preprint arXiv:1905.09883 (2019).

\bibitem[YW71]{Yamada71}
Toshio Yamada and Shinzo Watanabe, \emph{{On the uniqueness of solutions of
  stochastic differential equations}}, Journal of Mathematics of Kyoto
  University \textbf{11} (1971), no.~3, 155 -- 167.

\end{thebibliography}
\bibliographystyle{amsalpha}
	
\end{document}